\documentclass[12pt]{amsart}
\usepackage{amscd}
\usepackage[all]{xy}

\newcommand{\iso}{\overset{\sim}{\longrightarrow}}
\newcommand{\C}{\mathbf{C}}

\renewcommand{\P}{\mathbf{P}}
\newcommand{\Q}{\mathbf{Q}}
\newcommand{\Z}{\mathbf{Z}}
\newcommand{\un}{\mathbf{1}}

\newcommand{\sM}{\mathcal{M}}

\newcommand{\sJ}{\mathcal{J}}

\newtheorem{thm}{Theorem}
\newtheorem{lm}{Lemma}
\newtheorem{cor}{Corollary}

\theoremstyle{definition}
\newtheorem{defn}{Definition}

\newtheorem{prop}{Proposition}
\newtheorem{rk}{Remark}

\begin{document}

\title{On  the finite dimensionality of a K3 surface}
\author{Claudio Pedrini}  
 \maketitle
  Universit\'a di Genova,  Dipartimento di Matematica, Via
Dodecaneso 35,  16146 Genova( Italy),  \email{pedrini@dima.unige.it } \bigskip

{\it Abstract}:  For a smooth projective surface $X$ the finite dimensionality   of the Chow motive $h(X)$,  as conjectured by S.I Kimura, has several geometric consequences. For a complex surface of general type with $p_g=0$ it is equivalent to Bloch's conjecture. The conjecture is still open for a K3 surface $X$ which is not a Kummer surface. In this paper we prove some results on Kimura's conjecture for complex K3 surfaces. If  $X$ has a large Picard number $ \rho = \rho(X)$,i.e $\rho =19,20$,  then the motive of $X$ is finite dimensional.  If  $X$  has  a non-symplectic group acting trivially on algebraic cycles  then  the motive of $X$ is finite dimensional.  If $X$ has a symplectic involution $i$, i.e a Nikulin involution,  then the finite dimensionality of $h(X)$ implies $h(X)\simeq h(Y)$, where $Y$ is a desingularization of the quotient surface $X/<i>$.We   give several examples of K3 surfaces with a Nikulin involution such that the isomorphism $h(X) \simeq h(Y) $ holds, so giving some evidence to Kimura's  conjecture in this case. \bigskip

 \section {Introduction}
 For a smooth projective variety  $X$  over a field $k$  we will denote by $A^i(X)$ the Chow group of codimension $i$ cycles with rational coefficients and by   $\sM_{rat}(k)$ the (covariant ) category of Chow motives  with rational coefficients over the field $k$, which is  is a  $\Q$-linear, pseudoabelian, tensor category.\par  
\noindent An  object $M \in \sM_{rat}(k)$ is of the form  $M =(X,p,m)$,  where $X$ is a smooth projective variety over $k$, $p$ a correspondence in $X \times X$ such that $p^2=p$   and $m \in \Z$. We will denote by $h(X)$ the motive $(X,\Delta_X,0)$,where $\Delta_X$ is the diagonal in $X \times X$. If $X$ and $Y$ are smooth (irreducible) projective varieties over $k$ then
$$Hom_{\sM_{rat}(k)}(h(X),h(Y)=A^{dim X}(X\times Y)$$
\noindent where $A^*(X \times Y) =CH^*(X \times Y) \otimes \Q$.\par   
\noindent We will consider a classical Weil cohomology theory $H^*$ with coefficients in a field $K$ of characteristic 0  which induces a tensor functor $H^* :\sM_{rat}  \to Vect^{gr}_K $ such that $H^i((X,p,m) = p^*H^{i -2m}(X,K)$ (see [KMP 1.4].  If $char \ k =0$ homological equivalence does not depend on the choice of $H^*$. By replacing rational equivalence with homological equivalence we get the category $\sM_{hom}(k)$ of homological motives.\par
\noindent For an  object  $M \in \sM_{rat}(k)$, one   defines   the exterior power $\wedge^n M \in \sM_{rat}(k)$  ( and similarly in $\sM_{hom}$) and the symmetric power $S^nM$ . (see [Ki]).  A motive $M$ is {\it finite dimensional } if it can be decomposed as $M =M^{+} \oplus M^{-}$ with $M^{+}$   evenly finite dimensional , i.e such that $\wedge^n M =0$ for some $n >0$ and $M^{-}$  oddly finite dimensional,  i.e such that  $S^nM=0$ for $n >0$.\par
\noindent    S.I.Kimura and O'Sullivan  ( se [Ki]) have conjectured that all the motives in $\sM_{rat}(k)$ are finite dimensional. The conjecture is  known for curves, for  abelian varieties  and for some surfaces:  rational surfaces, Godeaux surfaces, Kummer surfaces, surfaces with $p_g=0$ which are not of general type, surfaces isomorphic to a quotient $(C \times D)/G$, where $C$ and $D$ are curves and $G$ is a finite group. It is also known for Fano 3-folds (see[G-G]). In all these  known cases the motive $h(X)$ lies in the tensor subcategory of $\sM_{rat}(k)$ generated by abelian varieties.\par 
\noindent  If $M=h(X)$ is the motive of a surface then  the finite dimensionality of $M$ is equivalent to the vanishing of $\wedge^n t_2(X)$ for some $n >0$, where $t_2(X)$ is the {\it transcendental part}  of $h(X)$. This follows  from the existence of a refined Chow-K\"unneth decomposition for the motive  $h(X)$ of a surface  

$$h(X) = \un \oplus h_1(X) \oplus h^2_{alg}(X) \oplus  t_2(X) \oplus h_3(X) \oplus \mathbf L^2 $$

\noindent where  $ \un$ is the motive of a point  and $\mathbf L$ is the Lefschetz motive. (see [KMP]). In the above decomposition all the summands,   but possibly $t_2(X)$,   are finite dimensional  because they lie in the subcategory of $\sM_{rat}(k)$ generated by abelian varieties.Therefore the information necessary to study the above conjecture  
for a surface $X$ is concentrated in the transcendental part of the motive $t_2(X)$. More precisely,  according to  Murre's Conjecture (see [Mu]), or equivalently to Bloch-Beilinson's conjecture ( see [J])  and to
Kimura's Conjecture the following results should hold for a surface $X$  

(a) The motive $t_2(X)$ is evenly finite dimensional;\par 

(b) $h(X)$ satisfies the Nilpotency conjecture , i.e every homologically trivial endomorphism of $h(X)$ is nilpotent ;\par  

(c) Every homologically trival correspondence in $CH^2(X \times X)_{\Q}$ acts trivially on the Albanese kernel $T(X)$; \par

d) The endomorphism group of $t_2(X)$ (tensored with $\Q$) has finite rank (over a field of
characteristic 0).\par 
\noindent By a result of S.Kimura in [Ki] , (a) implies (b).\par 
\noindent If $X$ is a complex surface of general type with $p_g(X)=0$ ,  Bloch's conjecture asserts that $A_0(X) \simeq \Q$. Then
$$ (a) \Longleftrightarrow A_0(X) =\Q \Longleftrightarrow t_2(X) =0$$
\noindent (see [G-P]).\par
\noindent  A case where all the above conjectures are still unknown is that of a complex K3 surface which is not a Kummer surface.\par
\noindent  The aim of this paper  is to prove some results about the finite dimensionality of $h(X)$ in the case $X$ is a K3 surface over $\C$.\par
\noindent  Note that  a result  by Y.Andre' in [A 10.2.4.1]   implies  that the motive of a K3 surface is isomorphic to the motive of an abelian variety in a suitable  category of {\it motivated motives}.  Under the standard  conjecture $B(X)$ this category coincides with $\sM_{hom}$ (see [A  p.100].  Therefore Andre's result suggests that the Chow motive of every K3 surface can be expressed in terms of the motives of abelian varieties.\par
In \S 2 we consider the case of a projective surface $X$ with an involution $\sigma$ and the desingularization  $Y$ of the quotient surface $X/<\sigma>$. Corollary 1 gives necessary and sufficient conditions on $\sigma$ for  the existence of an isomorphism    $t_2(X) \simeq t_2(Y)$ and for $t_2(Y)=0$. In particular  this result applies to  a   complex surface  of general type $X$ with $p_g(X)=0$ and an involution $\sigma$  for which $t_2(Y)=0$.\par
In \S 3  we   apply the results  in \S 2 to the case of a complex K3 surface $X$ with an involution $\sigma$. If $\sigma$ is symplectic , i.e 
$\sigma$  is a {\it Nikulin involution},  then   the finite  dimensionality of $h(X)$ implies the isomorphism $h(X)  \simeq h( Y)$, see Theorem 3.  If   the rank of the Neron- Severi group of $X$ is 19 or 20,  then $h(X)$ is finite dimensional (Theorem 2)  . If $\sigma$ is not symplectic then $t_2(Y) =0$, with $Y =X/<\sigma>$, hence $t_2(X) \ne t_2(Y)$, see Remark 3. If a K3 surface $X$  has  a non-symplectic group acting trivially on algebraic cycles  then   the motive of $X$ is finite dimensional ( Corollary 2).  Note that,  in all the cases where we can show that the motive $h(X) $ of a K3 surface is finite dimensional, $h(X)$ lies  in the tensor subcategory of $\sM_{rat}(k)$ generated by abelian varieties.\par
\noindent  In \S 4 , using the results  in [VG-S],  we describe   several examples of K3 surfaces , with a Nikulin involution $i$ and Picard rank 9, such that $t_2(X) \simeq t_2(Y)$.   We also show (see Theorem 7) that  the same result holds if the  K3 surface $X$  has an elliptic fibration $ X \to \P^1$ with a section. This gives some evidence to Kimura's conjecture for a K3 surface with a symplectic involution.
\par \medskip

We thank A.Del Padrone, V.Guletskii, B.Kahn and C.Weibel for many helpful comments on a earlier draft of this paper.We also thank the Referee for suggesting several improvements   in the exposition and simplifications in the proofs.

\section { Surfaces with an involution}

In this section  we prove some results on the transcendental part $t_2(X)$ of the motive of  a surface $X$,with an involution $\sigma$.\par
\noindent We first note that, if $X$ is a smooth projective variety over a field $k$,  and  $G$ is a finite group  acting on $X$, then  the theory of correspondences can be extended to $Y =X/G$,  if one uses rational coefficients in  the Chow groups  ( see [Fu 16.1.13]).   In particular this holds  if  $G =<\sigma>$, where $\sigma$ is an involution.\par
\noindent Let $X$ be  a smooth irreducible projective surface (over any field  $k$) with a refined Chow-K\"unneth decomposition  $\sum_{0 \le i \le 4}h_i(X)$
where  $h_2(X) =h^{alg}_2(X) +t_2(X)$ and $t_2(X)= (X,\pi^{tr}_2,0)$, see [KMP 2.2]. Here  
$$ \pi^{alg}_2(X)= \sum_{1 \le h \le \rho} \frac {[D_h \times D_h]}{D^2_h} $$
\noindent where $\{D_h\}$ is an  orthogonal basis of $NS(X) \otimes \Q$ and $\rho = rank \  NS(X)$.
The map 
\begin {equation}
 \Psi_X : A^2(X \times X)  \to  End_{\sM_{rat}}(t_2(X))
\end{equation} 
defined by  $\Psi_X(\Gamma) =\pi^{tr}_2 \circ \Gamma \circ \pi^{tr}_2$  yields  an isomorphism (see [KMP 4.3])
$$ A^2(X \times X)/\sJ(X) \simeq End_{\sM_{rat}}(t_2(X))$$
\noindent where   $\mathcal J( X)$ is the ideal of $A^2(X \times X)$ generated by the classes of correspondences which are not dominant over $  X$ by either the first or the second projection.   Let $k(X)$ be the field of rational functions and let  $T(X_{k(X)}))$ be  the Albanese kernel of $X_{k(X)}$, i.e the kernel of the Abel-Jacobi map $A_0(X_{k(X)}) \to Alb_X(k(X)\otimes \Q$. Let 
$\tau_X : A^2(X \times X) \to T(X_{k(X)}) $ be the map  
$$\tau_X(Z) = (\pi^{tr}_2 \circ Z \circ \pi^{tr}_2)(\xi)$$
\noindent with $\xi$ the generic point of $X$.  Then $\tau_X$  induces an  isomorphism  (see [KMP 5.10])
$$ End_{\sM_{rat}}(t_2(X))\simeq \frac{ T(X_{k(X)})} {H_{\le1}\cap T(X_{k(X)})}$$
\noindent Here $H_{\le1}$ is the subroup of $A_0(X_{k(X)})$ generated by the subgroups $A_0(X_L)$, where $L$ runs over the subfields of $k(X)$ containing $k$ and which are of transcendence degree $\le1$ over $k$. If $q(X)=0$ then $X$ has  no odd cohomology, $Alb_X(k)=0$  and  in the Chow -K\"unneth decomposition we have $h_1(X) =h_3(X) =0$. Therefore  $A_0(X_{k(X)})_0=  T(X_{k(X)})$ and  $T(X)= A_0(X)_0$, where $A_0(X)_0$ is the group of 0-cycles of degree 0. 
By [KMP 5.10 ] we have 
$$ H_{\le1} \cap  T(X_{k(X)}) = T(X)$$
\noindent Hence , for a  surface $X$  with $q(X)=0$,  the map $\tau_X$ yields an isomorphism
\begin{equation}
 End_{\sM_{rat}}(t_2(X)) \simeq \frac { A_0(X_{k(X)})} {A_0(X))}
\end{equation}
\noindent where the class $[\xi]$ in $ \frac { A_0(X_{k(X)})} {A_0(X))}$ of the generic point $[\xi]$ of $X$ corresponds to the identity of the ring $End_{\sM_{rat}}(t_2(X))$ \par
\noindent The definition of the map $\Psi_X$ in  (1) can be extended to the case of  two smooth projective surfaces $X$ and $X'$ as in [KMP 7.4]
$$\Psi_{X,X'} : A^2(X \times X') \to Hom_{\sM_{rat}}(t_2(X),t_2(X'))$$
\noindent  and the following functorial relation holds
\begin {equation} 
\Psi_{X,X"} ( \Gamma' \circ \Gamma) = \Psi_{X',X"} (\Gamma') \circ \Psi_{X,X'}(\Gamma)
\end{equation} where $X,X',X"$ are smooth projective surfaces, $\Gamma \in A^2(X\times X')$ and $\Gamma' \in A^2(X' \times X")$. The proof of (3) immediately follows
 by taking refined Chow-K\"unneth decompositions of the motives $h(X)$ , $ h(X')$ , $h(X")$ and writing the elements in $Hom_{\sM_{rat}}(h(X), h(X'))$, and $Hom_{\sM_{rat}}(h(X'), h(X"))$ as lower triangular  matrices defined by these decompositions, as   in [KMP p.163]. Applying $\Psi$ corresponds to taking appropriate diagonal entries of such lower triagular matrices.\par

\begin {lm} Let  $X$ and $Y$ are smooth projective surfaces and let $f: X \to Y$ be a finite morphism.Then  $f$ induces homomorphisms
$\bar f_* : End_{\sM_{rat}}(t_2(X)) \to End_{\sM_{rat}}(t_2(Y))$ and $\bar f^*: End_{\sM_{rat}}(t_2(Y) \to End_{\sM_{rat}}(t_2(X)$.
\end {lm}

\begin {proof} The maps $\Psi_X : A^2(X\times X) \to End_{\sM_{rat}}(t_2(X))$ and $\Psi_Y : A^2(Y\times Y) \to End_{\sM_{rat}}(t_2(Y))$
give rise to the following commutative  diagram  

$$ \CD  0 \to \sJ(X)@>>> A^2(X \times X)@>{\Psi_X}>> End_{\sM_{rat}}(t_2(X))@>>> 0 \\
  @VVV     @V {(f \times f)_*}VV   @V{\bar f_*}VV   @. \\
  0\to \sJ(Y)@>>> A^2(Y \times Y)@>{\Psi_Y}>>End_{\sM_{rat}}(t_2(Y))@>>> 0 \\
  @VVV       @V{(f\times f)^*}VV     @V{ \bar f^*}VV  @.  \\
  0 \to \sJ (X)@>>> A^2(X \times X)@>{\Psi_X}>>End_{\sM_{rat}}(t_2(X))@>>> 0 
 \endCD $$
\noindent where the map $(f \times f)_* $ sends a correspondence $Z \in A^2(X \times X)$   to $\Gamma_f \circ Z \circ \Gamma^t_f$ and the map $f\times f)^*$ sends  a correspondence 
$Z' \in A^2(Y \times Y) $ to $\Gamma^t_f \circ Z' \circ \Gamma_f$. It is easy to see that these maps send the ideal  $\sJ(X) $ to $\sJ(Y)$ and $\sJ(Y)$ to $\sJ(X)$ respectively, thus yielding the diagram above.
\end{proof}

\begin {prop} Let $X$ be a smooth projective surface  with an involution $\sigma$,  such that  the quotient surface   $Y= X/<\sigma> $ is smooth.  Let $\xi$ denote the generic point of $X$,$\eta$ the generic point of $Y$  and let $[\xi] =\Psi_X(\Delta_X) \in End_{\sM_{rat}}(t_2(X))$, $[\eta] = \Psi_Y(\Delta_Y) \in End_{\sM_{rat}}(t_2(Y))$. Set $\alpha = \Psi_X(1\times \sigma)\Delta_X = \Psi_X(\Gamma_{\sigma})=\bar\sigma([\xi])$. Then the map $f : X \to Y$ satisfies:\par   
 (i)$ 1/2(\Gamma_f \circ \Gamma^t_f) = \Delta_Y$,  $ \bar f_*([\xi]) = \bar f_*(\alpha)= 2 [\eta] $ and  $(\alpha)^2 =[\xi] $.\par
 (ii) $\bar f^*([\eta]) = [\xi] +\alpha$ and $\bar f^*(\bar f_*([\xi]) = 2[\xi] +2\alpha$.\par
 (iii) Let $p = 1/2(\Gamma^t_f \circ \Gamma_f)$; then  $p\circ p =p$,  $\Psi_X(p) =1/2([\xi] +\alpha)$ and\par 
 $\Psi_X(\Delta_X - p) = 1/2([\xi] -\alpha)$.Hence  $[\xi] =1/2([\xi] +\alpha) +1/2([\xi] -\alpha).$ 
 
 \end{prop} 
\begin {proof} Regard the diagonals  $\Delta_X$ and $\Delta_Y$  as cycles in $A^2(X \times X)$ and $A^2(Y \times Y)$). Then $ \bar f_*([\xi])$ is the image under $\Psi_Y$ of $\Gamma _f \circ \Gamma^t_f =2\Delta_Y$. Thus 
$ \bar f_*([\xi])=2\Psi_Y(\Delta_Y) =2[\eta]$ and we also have
$$ p \circ p =(1/4)\Gamma^t_f \circ  (2\Delta_Y)\circ \Gamma_f = 1/2(\Gamma^t_f \circ \Gamma_f) =p$$
\noindent Since $\alpha$ is the image of $(1\times \sigma)\Delta_X$ , and $(1\times \sigma) \Delta_X \cdot (1\times \sigma) \Delta_X) =\Delta_X$ we have $\alpha^2 =[\xi]$. Since $\Gamma_f \cdot(1\times \sigma)\Delta_X =\Gamma_f$, the correspondence  $(1\times \sigma)\Delta_X$ also maps to $\Delta_Y$, so $f_*(\alpha) =2 [\eta]$. This establish (i) and (ii) follows immediately. Part (iii) follows from (ii) and $p \circ p =p$.
\end{proof}
\medskip

Let $X$ be a smooth projective surface  and let $\sigma $ be an involution on $X$. Let  $k$ be the number of isolated fixed points of $\sigma$ and let $D$ the 1-dimensional part of the fixed-point locus. The divisor $D$ is smooth (possibly empty). Let $\tilde X$ be the blow-up of of the set of isolated fixed points . Then the involution $\sigma $ lifts to an involution on $\tilde X$  (which we will still denote by $\sigma$). The quotient $Y = \tilde X/<\sigma>$  is a desingularization of $X/<\sigma>$. $Y$
has $k$ disjoint nodal curves $C_1,\cdots ,C_k$. The map $ X \to X/<\sigma>$ induces a commutative diagram
  $$ \CD \tilde X @>{\beta}>>X \\
@V{f}VV   @VVV \\
 Y@>>> X/<\sigma> \endCD  $$
Since $t_2(-)$ is a birational invariant for smooth projective surfaces the maps $\beta : \tilde X \to  X$ and $f : \tilde X \to Y$ induce a morphism
$$ \theta : t_2(\tilde X)=t_2(X) \to t_2(Y)$$

\begin {cor} Let $X,\tilde X, Y$ be as in the diagram above. Then \par
 (i)  $\theta : t_2(X) \to t_2(Y)$ is the projection onto a direct summand.\par
 (ii) $\theta$ is an isomorphism  iff $\Psi_X(\Gamma_{\sigma})= id_{t_2(X)},  $ i.e iff   $\bar \sigma([\xi])=[\xi]$ in $End_{\sM_{rat}}(t_2(X))$.\par
 (iii) If $q(X)=0$ the conditions of (ii) are equivalent to $A_0(X)^{\sigma}_0 =A_0(X)_0$.\par 
 (iv)  $ t_2(Y) =0  \Longleftrightarrow \Psi_X(\Gamma_{\sigma})= -id_{t_2(X)}  \Longleftrightarrow \bar \sigma([\xi])=- [\xi]$ in $ End_{\sM_{rat}}(t_2(X)$.  If $q(X)=0 $ this is equivalent to   $A_0(X)^{\sigma}_0 =0 $.
\end{cor}
 
\begin {proof} Since $\Psi_{\tilde X,X}(\Gamma_{\beta})$ is an isomorphism and $\theta =\Psi_{\tilde X,Y}(\Gamma_f) \circ \Psi_{\tilde X ,X}(\Gamma_{\beta})^{-1}$\par
\noindent it is enough, after replacing $X$ by $\tilde X$,  to prove the Corollary under the assumption $\tilde X =X$. Then $\theta = \Psi_{X,Y}(\Gamma_f)$.  From Proposition 1 we get  that $\Gamma_f $ has a right inverse $1/2(\Gamma^t_f )$  and $ 2p =\Gamma^t_f \circ \Gamma_f=  \Delta_X +(1 \times \sigma)\Delta_X = \Delta_X + \Gamma_{\sigma}$.  It follows from the functoriality of $\Psi$ in (3) that , if $t_2(X)^{+}$ and $t_2(X)^{-}$ are the direct summands of $t_2(X)$ on which the involution $\Psi_X(\Gamma_{\sigma})$ acts respectively as $+1$ or $-1$,  then the restriction of $\theta$ to $t_2(X)^{-}$ is 0 and to $t_2(X)^{+}$ is an isomorphism. This gives (i).\par
\noindent  Also $\theta $ is an isomorphism iff $t_2(X)^{-}=0$ which is equivalent to $\Psi_X(\Gamma_{\sigma})$ being the identity in $End_{\sM_{rat}}(t_2(X))$. This gives (ii).\par
\noindent  If $q(X)=0$ then $A_0(X)_0=T(X)$ and we have a canonical isomorphism
$$ Hom_{\sM_{rat}}(\un , t_2(X)) \simeq A_0(X)_0$$
\noindent which is compatible with the action of correspondences. Hence , by taking the action of $\Psi_X(\Gamma_{\sigma})$ on $t_2(X$ we get

$$Hom_{\sM_{rat}}(\un , t_2(X)^{-}) \simeq A_0(X)^{-}_0 $$

\noindent Therefore $\Psi_X(\Gamma_{\sigma})$ acts as the identity on $t_2(X)$ iff $A_0(X)^{-}_0 = A_0(t_2(X)^{-})  =0$.  Since $A_i(t_2(X)) =0$ for $i\ne 0$,  we have 
 $A_i(M) =0$ for all  $i$,  where $M= t_2(X)^{-}$.  It follows that $M=0$ (see [C-G Lemma 1]. This proves (iii).\par
\noindent Clearly $t_2(Y) =0$ is equivalent to   $\bar \sigma([\xi])= - [\xi] \in End_{\sM_{rat}}(t_2(X)) $. Since the cycle class $[\xi]$ corresponds to the identity of  $End_{\sM_{rat}}(t_2(X))$ under the isomorphism in (2),  $\Psi_X(\Gamma_{\sigma})$ acts as $-1$ on $t_2(X)$.  Let $q(X)=0$ : then , by the same argument as in the proof of (iii) we get $A_0(X)^{\sigma}_0 =0$. This gives (iv).
\end{proof}

 F.Severi in  [Sev] has introduced the notions of    {\it valence } and    {\it indices } of a correspondence $T \in A^n(X \times X)$, where $X$ is  a smooth projective variety of  dimension $n$ .  In the case when $X$ is a surface,  Severi  related these notions to the computation of the degree of the cycle  $T \cdot \Delta_X$.\par
\begin {defn} Let $X$ be a smooth projective variety of dimension $n$. A correspondence   $T \in A^n(X \times X)$  has  valence 0 if it belongs to the ideal   of degenerate correspondences , i.e the ideal generated by  correspondences of the form  $[ V \times W]$, with $V,W$ proper  subvarieties of $X$.  A correspondence $\Gamma $ has valence $v$  if  $T =\Gamma +v \Delta_X$ has valence 0. If $T = T_1 +T_2 $ in $A^d(X \times X)$ and $T_1 ,T_2$ have valences $v_1 ,v_2$ then $T$ has valence $v_1 +v_2$ . If the correspondences $T$ and $T'$ have valences $v$ and $v'$ then  $v(T \circ T' )=-v(T) \cdot v(T')$; see[Fu 16 1.5.].  It follows that if , $p$ is a projector in $A^2(X \times X)$ which has a valence,  then $v(p)$ is either 0 or -1.\par
The indices of a correspondence $T$ are the numbers   $\alpha(T) =deg(T \cdot [P \times X])$ and  $\beta (T) = deg (T \cdot [X \times P])$, where $P$ is any rational point on $X$; see [Fu, 16 1.4]. The indices are additive in $T$ and $\beta(T) = \alpha (T^t)$.
\end{defn}

 \begin {thm} Let $X$ be a smooth projective surface with an involution $\sigma$  and let  $Y$ be the desingularization of  $X/<\sigma>$. Assume $p_g(X) >0$ and let $\Gamma_{\sigma} =  (1 \times \sigma)\Delta_X) $ . If the correspondence $\Gamma_{\sigma} $ has a valence, then
$$ t_2(Y) =0    \Longleftrightarrow v(\Gamma_{\sigma} ) =1     \    ;   \  \theta :  t_2(X) \iso t_2(Y)  \Longleftrightarrow  v(\Gamma_{\sigma}) =-1$$
\end{thm}
\begin{proof}:   Let  $[\xi] = 1/2([\xi] +\alpha)  + 1/2([\xi] -\alpha)$ be the splitting  in $End_{\sM_{rat}}(t_2(X))$ coming from Proposition 1, with $\alpha = \bar \sigma([\xi])= \Psi_X(\Gamma_{\sigma})$.  If the correspondence $\Gamma_{\sigma}  $ has a  valence  then also the projector $p= 1/2(\Delta_X +(1\times \sigma)\Delta_X) = 1/2(\Delta_X +\Gamma_{\sigma})$ has  a valence and $v(p)$ is either 0 or -1.
Since $v(\Delta_X )=-1$ we have
$$ v(p) =0  \Longleftrightarrow v(\Gamma_{\sigma}) =1 \  ; \ v(p) =-1 \Longleftrightarrow v(\Gamma_{\sigma})= -1$$
\noindent Suppose $v(\Gamma)=1$ : then $v(p)=0 $ i.e $p$ belongs to the ideal of degenerate correspondences , which is contained in $Ker \  \Psi_X$. From $\Psi_X(p)=0$ we get  
 $ 1/2([\xi] +\Psi_X(\Gamma_{\sigma})) =0$ hence $\Psi_X(\Gamma_{\sigma})= - id_{t_2(X)}$. From Corollary (iv) we get $t_2(Y)=0$. Conversely if $t_2(Y)=0$, then  $\Psi_X(\Gamma_{\sigma})= - id_{t_2(X)}$ hence $\Psi_X(p)=0$ and we get $v(p) =0$.\par
 \noindent If $\theta :t_2(X) \to t_2(Y) $ is an isomorphism then, by Corollary 1 (ii)   $\Psi_X(\Gamma_{\sigma}) = id_{t_2(X)}$. By the same argument as before we get $v(\Gamma_{\sigma})= -1$. 
 \end{proof}

\begin {rk}The assumption $p_g(X)>0$ in Theorem 1 is necessary in order to have a uniquely defined valence for $\Gamma_{\sigma} $.  If $p_g(X)=0$ and $X$ satisfies Bloch's conjecture then , by the results in [B-S] , $v(\Delta_X)=0$, hence the correspondence  $\Delta_X$ has 2 different valences;   namely $0$ and $-1$.  Note that,   for a surface $X$,  a  correspondence $\Gamma$ can have 2 different valences $v$ and $v'$ only if $p_g(X)=0$ .This   was first observed by Severi in [Sev p.761]). In fact  then the   multiple $(v-v')$ of the diagonal $\Delta_X$ belongs to the ideal of degenerate correspondences  and this implies that $\Psi_X(\Delta_X) = 0$ in $End_{\sM_{rat}}(t_2(X))$.  Therefore 
the identity map is 0 in $End_{\sM_{rat}}(t_2(X)$.  Hence $t_2(X)=0$ and this may occur only if $p_g(X)=0$.
\end{rk}

\section { Complex K3 surfaces } 
 
  A smooth (irreducible) projective  K3 surface  $X$ over $\C$ is a regular surface (i.e $q(X)=0$),  therefore it has  a  refined Chow-K\"unneth decomposition ( see [KMP 2.2]) of the form
$h(X)=\sum_{0 \le i\le 4}h_i(X)$ with $h_1(X)=h_3(X)=0$. Moreover  $h_2(X) =h^{alg}_2(X) +t_2(X)$, where $t_2(X)= (X,\pi^{tr}_2,0)$ and 
$h^{alg}_2(X) \simeq \mathbf{L}^{\oplus \rho(X)}$. Here $\rho(X)$  is  the  rank of the $NS(X)_{\Q} = (Pic X)_{\Q}$ so that  $1 \le  \rho  \le
20$. Moreover   
$$H^i(t_2(X))=0  \   for \   i \ne 2    \  ;   \  H^2(t_2(X)) =\pi^{tr}_2 H^2(X,\Q) =H^2_{tr}(X,\Q),$$  
$$A_i(t_2(X))=\pi^{tr}_2 A_i(X)=0    \  for  \    i \ne 2   \  ; \  A_0(t_2(X)) =T(X),$$  
\noindent where $T(X)$ is the Albanese Kernel. Since $q(X) =0$, we also have $T(X) =A_0(X)_0$ (0-cycles of degree 0) and
$$ dim H^2(X) = b_2(X) =22  \ ; \  dim H^2_{tr}(X) = b_2(X)  - \rho  $$
A Nikulin involution $i$ of  a complex K3 surface $X$ is a symplectic  automorphism of order 2 , i.e such that $i^*\omega=\omega$ for all $\omega\in H^{2,0}(X)$.  A  K3 surface $X$ with a Nikulin involution has  rank $\rho(X) \ge 9$ .The Neron-Severi group $NS(X)$ contains a primitive sublattice isomorphic to $E_8(-2)$ where $E_8$ is the unique even unimodular positive  defined lattice of rank 8 (see [Mor p.106]). Here, if $L$ is a lattice and $m$ is an integer, $L(m)$ denotes same  free $\Z$-module $L$ with a form which has been altered by multiplication by $m$, that is $b_{L(m)}(x,y) = m(b_{L}(x,y))$, where $b_{L}(x,y)$ is the $\Z$-valued  symmetric bilinear form of $L$. By $T_X$ we will denote the transcendental lattice of $X$,  i.e $T_X =NS(X)^{\perp} \subset H^2(X,\Z)$.
For any K3 surface with a Nikulin involution $i$ there is an  isomorphism
$$H^2(X,\Z) \simeq U^3 \oplus E_8(-1) \oplus E_8(-1)$$
\noindent where $U$ is the hyperbolic plane, such that $i^*$ acts as follows 
  $$i^*(u,x,y)=(u,y,x)$$
\noindent The invariant sublattice is $H^2(X,\Z)^{i} \simeq U^3 \oplus E_8(-2)$ and 
$(H^2(X,\Z)^i)^{\perp} \simeq E_8(-2)$.  Since $i^*\omega= \omega$ for all $\omega \in H^{2,0}(X)$ we also have $(H^2(X,\Z)^i)^{\perp} \subset NS(X)$ ( see [VG-S 2.1]).
Therefore the involution $i$ acts as the identity on $H^2_{tr}(X,\Q)$.\par
\noindent  Let $ X \to X/<i> $ be the quotient map. The surface $X/<i>$ has $8$ ordinary double points $Q_1,\cdots,Q_8 $ corresponding to the 8 fixed points $P_1,\cdots,P_8$ of the involution $i$ on $X$. The minimal model $Y$ of $X/<i>$  is a K3 surface, hence $p_g(Y) >0$.  In the following we will  always consider the {\it standard diagram} for a K3 surface with a Nikulin involution $i$ (see [Mor sec. 3])
$$ \CD \tilde X @>{\beta}>>X \\
@V{f}VV   @VVV \\
 Y@>>> X/<i> \endCD  $$
\noindent $\tilde X$ is the blow up of $X$ at the points $P_1,\cdots,P_8$ with exceptional divisors $\beta^{-1}(P_j) =E_j$. The Nikulin involution  extends to an involution $i$ 
on $\tilde X$ and $Y =\tilde X/<i>$. $ f $ is a double cover branched on the divisor $\sum_{1 \le j \le 8}C_j$ where 
$C_j =f(E_j)$ are disjoint smooth irreducible rational curves corresponding to the points $Q_1,\cdots,Q_8$. Therefore $1/2(\sum_j C_j) \in NS(Y)$. The map $f_* \circ \beta^*$ induces an isomorphism of rational Hodge structures
$$T_{\tilde X}\otimes \Q \simeq T_X \otimes \Q \simeq T_{Y}\otimes \Q$$
\noindent where $T_X$ and $T_{ Y}$ are the transcendental lattices . In particular  the vector spaces $H^2_{tr}(X, \Q)$ and $H^2_{tr}( Y, \Q)$ have the same dimension, so that $22 -\rho(X) = 22 -\rho(Y)$.\par
Suppose conversely that a K3  surface $Y $admits an even set of $k$ disjoint rational curves $C_1,\cdots ,C_k$ : this means that there exists a $\delta \in Pic Y$ such that
$$C_1 +\cdots+C_k \sim 2 \delta$$
This is equivalent to the existence of a double cover  $X$ of $Y$ branched on $C_1 +\cdots +C_k$. Then,  by [N 1] , $k =0,8,16$. If $k =16$ then $X$ is birational to an abelian surface $A$ and $Y$ is the Kummer surface of $A$. Therefore the motives $h(X)$ and $h(Y)$ are finite dimensional and $t_2(A) \simeq t_2(X) \simeq t_2(Y)$ (see [KMP 6.13]).
If $k =8$ then $X$ is a K3 surface  ,  $Y$ is the desingularization of the quotient of  $X$ by a Nikulin involution $i$. Hence $Y$ is a K3 surface.\par  \medskip
 \begin {thm} Let $X$ be a smooth projective K3 surface over $\C$ with $\rho(X) =19,20$. Then the motive $h(X) \in \sM_{rat}(\C)$ is finite dimensional and  lies in the subcategory  of $\sM_{rat}(\C)$ generated by the motives of abelian varieties.
 \end{thm} 
\begin {proof}  By [Mor 6.4] $X$ admits a  Shioda-Inose structure, i.e there is a Nikulin involution $i$ on $X$ such that the desingularization $Y$ of the quotient surface $X/<i>$ is a Kummer surface, associated to an abelian surface $A$.  Hence $h(Y)$ is finite dimensional. The rational map $f : X \to Y$ induces a splitting $t_2(X) \simeq t_2(Y) \oplus N$. Since $t_2(Y) $ is finite dimensional  we are left to show that $N=0$ . From Corollary 1  the vanishing of $N$ is equivalent to $A_0(X)^i_0 =A_0(X)_0 $.   By  [Mor 6.3 (iv)]  the Neron -Severi group of $X$ contains the sublattice $E_8(-1)^2$ .  Hence by the results in [Huy 6.3 , 6.4], the  symplectic automorphism  $i$   acts as the identity on$A_0(X)$. From Corollary 1 we get  $t_2(X) =t_2(Y)$. By [KMP 6.13]  $t_2(Y) =t_2(A)$;  therefore $h(X) $ is finite dimensional and lies  in the subcategory  of $\sM_{rat}(\C)$ generated by the motives of abelian varieties.
\end{proof}
\begin {rk} Note that  by [Mo 2.10 (i) ,4.4(i)] there exist K3 surfaces with $\rho(X)= 19,20$ which are not Kummer surfaces. \end{rk}

Next we show that for every K3 surface with a Nikulin involution the finite dimensionality of $h(X)$ implies $h(X) \simeq h(Y)$.\par
\begin {lm} Let $X$ be a K3 surface over $\C$ with a Nikulin involution $i$and let  $ Y$ be a desingularization of the quotient surface $X/<i>$.Let  $e(-)$ be the topological Euler characteristic . Then we have
$$e(X) +t  +2 +2k = 2e(Y)$$
where $t$ is the trace of the involution $i$ on $H^2(X, \C)$  and $k=8$ is the number of  the isolated fixed points of $i$. Therefore   $\rho (X) =\rho (Y)$ and $t=6$ 
\end{lm}
\begin{proof}: We use the same argument as  in [D-ML-P 4.2] . Since $i$ has only isolated fixed points from the topological fixed point formula we get
 $$e(X) +t +2 = 2e(Y)- 2k$$
\noindent Since $X$ and $Y$ are  both K3 surfaces we have $e(X) =e(Y)  =24$. Therefore, we get $t =6$. Since $dim \ H^{tr}_2(X) = dim \ H^{tr}_2(Y)$ and $b_2(X)=b_2(Y) =22$,  we have $\rho(X) =\rho(Y)$
\end{proof}  

 \begin{thm}  Let $X$ be K3 surface with a Nikulin involution $i$. If $h(X)$ is finite dimensional then  $h( X) \simeq h( Y)$. \end{thm}
 
 \begin{proof}   $Y$ is a K3 surface and we have    $t_2( \tilde X) =t_2(X)$ because $t_2(-)$ is a
birational invariant for surfaces.Also 
$$H^2_{tr}(X) \simeq H^2_{tr}( \tilde X) \simeq  H^2_{tr}(Y)$$ 
 because the Nikulin involution acts trivially on $H^2_{tr}(X)$. 
Let $t$ be  the trace of the involution $\sigma$ on  the vector space $H^2(X, \C)$. From Lemma 2 we get $t=6$. The involution $i$ acts trivially on $H^2_{tr}(X)$ and $H^2_{tr}(X)$ is a subvector space of $H^2(X,\C)$ of dimension $22-\rho $.  Therefore the trace of the action of $i$ on $NS(X)\otimes \C$ equals $\rho-16$ . Since the only eigenvalues of  an involution are +1 and -1 we can find an orthogonal basis for $NS(X) \otimes \C$ of the form $H_1,\cdots H_r; D_1, \cdots D_8 $, with $r =\rho-8 \ge 1$  such that $i_*(H_j) =H_j$ and $i_*(D_l) =- D_l$. Then $NS(\tilde X)\otimes \C$ has a basis of the form $E_1,\cdots E_8;H_1,\cdots H_r; D_1, \cdots D_8$, where  $E_h$ , for $1 \le h \le 8 $ are the exceptional divisors of the blow up $\tilde X \to X$.  The set of  $r +8=\rho$ divisors $f_*(E_h)=C_k $,  for $1 \le h \le  8$ and $f_*(H_ j) \simeq H_j $, for $ 1 \le  j \le r$  gives an orthogonal basis for $NS(Y)\otimes \Q$. Since $q(X) =q(Y) =q(\tilde X)=0$ we can find Chow-K\"unneth decompositions  for $h(X)$, $h(\tilde X)$  and  $h(\tilde Y)$ of the form

$$h(X) =  \un \oplus  h^{alg}_2(X)  \oplus t_2(X) \oplus \mathbf L^2 \simeq \un \oplus \mathbf L^{\otimes \rho}\oplus t_2(X) \oplus \mathbf L^2 $$

$$ h(\tilde X) = \un \oplus h^{alg}_2(\tilde X) \oplus t_2(X) \oplus \mathbf L^2 \simeq  h(X) \oplus \mathbf L^{\oplus 8}$$ 

$$ h(Y) = \un \oplus h^{alg}_2(Y) \oplus t_2(Y) \oplus \mathbf L^2 \simeq   \un \oplus \mathbf L^{\otimes \rho}\oplus t_2(Y) \oplus \mathbf L^2 $$

 \noindent where $\rho(X)=\rho(Y)=\rho$. By Corollary 1 we have
$t_2(X) \simeq t_2(Y) \oplus N$ for some $N \in \sM_{rat}$. Since $h(X)$ is finite dimensional also $N$ is finite dimensional. Since $H (X)$ and $H (Y)$ are isomorphic as graded vector spaces $H(N)=0$. By [Ki 7.3]  $N=0$. Therefore $h(X) \simeq h(Y)$.
\end{proof}

 \begin {thm} 
 Let $X$ be a K3 surface with a Nikulin involution $i$. Then the following conditions are equivalent:\par
 (i)  the correspondence $ \Gamma_i = (1 \times i)\Delta_X)$ has a valence.\par   
 (ii)  $ \theta : t_2(X) \iso  t_2(Y)$.  \par 
 (iii) $\bar i([\xi]) =[\xi]$ in $End_{\sM_{rat}}(t_2(X)$. \par
 (iv) $i$ acts as the identity on $A_0(X)_0$
   \end{thm}
 \begin {proof} Let 
 $$ \Delta_X = 1/2 (\Delta_X + (1\times i)\Delta_X) + 1/2 (\Delta_X -(1\times i)\Delta_X)$$
 as in Proposition 1 and let $\Gamma_i = (1\times i)\Delta_X $.  If  $\Gamma_i$ has a valence then also the projector $ q=1/2 (\Delta_X -(1\times i)\Delta_X)$ has a valence
 and $v(q)$ is either 0 or -1. Suppose that    $v (1/2 (\Delta_X -(1\times i)\Delta_X))= -1$;  then  the correspondence  $ (1\times i)\Delta_X)$ has valence 1.  From Theorem 1 we get 
  $t_2((Y) =0$ hence a contradiction because $Y$ is  a K3 surface. Therefore  $v (1/2 (\Delta_X -(1\times i)\Delta_X))= 0$,  so that    $v(1 \times i)\Delta_X)= v(\Delta_X) = -1$.   By Theorem 1  $\theta : t_2(X) \to t_2(Y)$ is an isomorphism. Therefore  $(i) \Rightarrow (ii)$.\par 
\noindent Conversely if $ \theta : t_2(X) \iso  t_2(Y)$  then  by Corollary 1 (ii) $\Psi_X( \Delta_X -\Gamma_i) =0$, hence $\Delta_X -\Gamma_i \in Ker \  \Psi_X$. Since $q(X)=0$
$Ker \  \Psi_X$ is coincides with the ideal of degenerate correspondences. This proves (i).\par 
\noindent  The equivalences  $(ii) \Leftrightarrow(iii)$ and $(iii) \Leftrightarrow(iv)$ come  from Corollary 1.
\end{proof} 
\begin {rk} Let  $\sigma$ is an involution on   a K3 surface $X$ which is not symplectic, i.e $\sigma^*(\omega) =-\omega$, where $\omega$ is a generator of the vector space $H^{2,0}(X)$.  By [Zh 1.2 ]  if $X^{\sigma}=\emptyset$ the quotient surface $Y=X/<\sigma>$ is an Enriques surface,   while  $Y$ is a rational surface  if $X^{\sigma} \ne \emptyset$.  In any case the motive $h(Y)$ has no transcendental part. Therefore  $t_2(Y)=0$ and   $t_2(X) \ne t_2(Y)$, because $t_2(X) \ne 0$ for a K3 surface.
 From the identity in $ End_{\sM_{rat}}(t_2(X)$
$$[\xi] =1/2([\xi] +\bar \sigma([\xi]) +1/2([\xi] -\bar \sigma([\xi])$$
\noindent we get
 $$  \bar \sigma(([\xi]) = - [\xi] \  {\text and}  \   [\xi] \ne -[\xi] \  in   End_{\sM_{rat}}(t_2(X)$$
\noindent  because otherwise we would get      $t_2(X)= 0$. Hence Theorem 4 does not hold true. 
  \end{rk}
\medskip 
Following the example in Remark 3  we now  consider the case of a complex K3 surface $X$ with a non-symplectic group $G$ acting trivially on the algebraic cycles.  Any automorphism $g $ of $X$ preserves the 1-dimensional vector space  $H^{2,0} = H^2(X,\Omega^2_X)\simeq \C \omega$.  Hence $g$ is non-symplectic iff there exists a   complex number  $\alpha(g) \ne 1$ such that$g^*(\omega) = \alpha(g)\omega$. Let $NS(X)$ and $T_X$ be the lattices of algebraic and transcendental  cycles on $X$. $X$ is said to be unimodular if $det \ T_X =\pm1$. Let $H_X$ be the finite  cyclic group defined as the kernel of the map $Aut(X) \to O(NS(X))$, where $O(NS(X))$ denotes the  group of isometries of $NS(X)$. Then there are only finitely many values for  $m =|H_X|$. By  [LSY Th. 5]  one has the following result.
\begin {thm}  Let $X$ be a  complex K3 surface $X$ with a non-symplectic group $G$ acting trivially on the algebraic cycles. Let $m =|H_X| \ne 3$ : then there exists a surjective morphism $ F_n  \to X$, where $F_n \subset \P^3$ is the Fermat surface, of degree $n \ge 4$  
$$ F_n : X^n_0 +X^n_1 +X^n_2 +X^n_3 = 0 $$ 
\noindent Here  $n =m$  if $X$ is unimodular and $n = 2m$,  if $X$ is not unimodular.
\end{thm}
\begin {cor} Let $X$ be a  complex K3 surface  with a non-symplectic group $G$ acting trivially on the algebraic cycles. Let $m =|H_X| \ne 3$. Then the motive of $X$ is finite dimensional and  lies in the subcategory  of $\sM_{rat}(\C)$ generated by the  motives of abelian varieties. K3 surfaces satisfying these conditions have 
$\rho(X ) = 2, 4, 6, 10, 12, 16, 18, 20.$ 
\end{cor}
\begin {proof} From Theorem 5  there is surjective morphism $F_n \to X$ ,with $F_n $ a Fermat surface. By  [SK] the motive  $h(F_n)$  is finite dimensional and lies in the subcategory  of $\sM_{rat}(\C)$ generated by the motives of abelian varieties. By [Ki 6.6 and 6.8]  if $f :Z \to X$ is a surjective  morphism of smooth  projective varieties, then $h(X)$ is a direct summand of $h(Z)$.Therefore  $h(X)$ is finite dimensional and lies in the subcategory  of $\sM_{rat}(\C)$ generated by the  motives of abelian varieties.  The computation of the rank $\rho(X)$  appears in [LSY Th. 1 and  Th. 2].
\end{proof} 
\medskip
  
 \section { Examples }
In this section we   describe some   examples of K3 surfaces with a Nikulin involution  $i$ ,  such that  $t_2(X) \simeq t_2(Y)$. Hence $h(X) \simeq h(Y)$.   We will use  the classification given by  Van Geemen and Sarti in [VG-S] and by Garbagnati and Sarti in [G-S]. Their results are based  on the following Theorem. 
 \begin{thm}([VG-S 2.2]) Let $X$ be K3 surface with $\rho(X) =9$ and  a Nikulin involution $i$. Let  $L$ be a generator of $E_8(-2)^{\perp}  \subset  NS(X)$ with $L^2=2d>0$ which we may assume to be ample.
Let 
 $$\Lambda_{2d} =\Z L \oplus E_8(-2)$$
 Then,  if $L^2\equiv 2 $ mod 4,  we have $\Lambda_{2d}=NS(X)$. If $L^2 \equiv 0 $ mod 4 we have either $NS(X) \simeq \Lambda_{2d}$ or $NS(X) \simeq \Lambda_{\overline {2d}}$.  Here
$  \Lambda_{\overline{ 2d}}$ is the unique even lattice containing $\Lambda_{2d}$ with  $\Lambda_{\overline{2d}}/\Lambda_{2d} \simeq \Z/ 2 \Z$ and such that $E_8(-2)$ is a primitive sublattice of 
$  \Lambda_{\overline{ 2d}}$. For every $\Gamma =\Lambda_{2d}$, with $d >0$ or $\Gamma =  \Lambda_{\overline{ 2d}}$ with $d =2m >0$,  there exists a K3 surface with a Nikulin involution $i$ such that
$NS(X) =\Gamma$ and $(H^2X,\Z)^i)^{\perp} \simeq E_8(-2)$.   
   \end{thm} 
Let's consider the following cases described in[VG-S]:  

 (i) $X$ is  a double cover of $\P^2$ branched over a sextic curve and $Y$ a a double cover of a quadric cone in $\P^3$;\par

(ii) $X$ is a double cover of a quadric in $\P^3$ and $ Y$is the double cover of $\P^2$ branched over a reducible sextic;\par

(iii) the image of $X$ under the map $\Phi_L$ is the intersection of 3 quadrics in $\P^5$ and $ Y$ is a quartic surface in $\P^3$.\par  

 First we l show  that in  the cases   (i),(ii)and (iii)  the map  $f:  X \to Y$ induces an isomorphism  

 $$t_2(X) \simeq  t_2( Y)$$

Then, in Theorem 7  we prove  that the same result holds if   $ g : X \to \P^1$ is a general elliptic fibration with a section and also $Y$ is an elliptic fibration.\par  
 In the case (i)  $NS(X) \simeq \Z L \oplus E_8(-2)$, with $L^2 =2$ and $i^*L \simeq L$ (see[VG-S 3.2]). The map $\Phi_L:  X \to  \P^2$  is a  double cover  branched over a sextic curve $C$ and   $ X/<i>$ is a double cover of a quadric cone in $\P^3$.  Let $\sigma$ denote the covering involution on $X$. Then $\sigma \ne i $. The quotient  surface $ Y= X/<\sigma>$ is isomorphic to $\P^2$.  Let $j= \sigma \circ i = i \circ \sigma$ and let  $G = <1,\sigma,i, j> \simeq \Z/2\Z \times \Z/2\Z$. The quotient  surfaces $\P^2 =X/<\sigma>$ and $S=X /<j>$  are both rational,  because $S$ is a Del Pezzo surface of  degree 1(see [VG-S 3.2]).  The motives $h(\P^2)$ and $h(S)$ have no transcendental part.  Therefore   $t_2(X) \ne t_2(Y)$  and $t_2(X) \ne t_2(S)$, because $t_2(X) \ne 0$ for a K3 surface.  From the identities  in $End_{\sM_{rat}}(t_2(X))$ 
$$[\xi] =1/2([\xi] +\bar \sigma([\xi]) +1/2([\xi] -\bar \sigma([\xi])$$
$$[\xi] =1/2([\xi] +\bar j[\xi]) +1/2([\xi] -\bar  j([\xi])$$
\noindent we get , by Corollary 1 (iv),   $[\xi] +\bar \sigma([\xi] )= [\xi] +\bar j([\xi])=0 $ in $ End_{\sM_{rat}}(t_2(X))$.\par
\noindent We have

$$\Psi(\Gamma_i) = \Psi_X( (1\times i)\Delta_X) = \Psi_X((1\times \sigma \circ j)\Delta_X)= \Psi(\Gamma_{\sigma\circ j})$$

\noindent Therefore the class of $\bar i ([\xi]) $ in  $End_{\sM_{rat}}(t_2(X))$ equals $\bar\sigma([\xi])\circ \bar j([\xi]) =  (- [\xi] )\circ  (-[\xi] ) = ([\xi])^2 =[\xi]$ because $[\xi]$ is the identity of $End_{\sM_{rat}}(t_2(X))$. Hence

$$ \bar i( [\xi])  - [\xi ]=0 \  in \   End_{\sM_{rat}}(t_2(X))$$

\noindent From Corollary 1 (ii)   we get $\theta : t_2(X) \iso t_2( Y)$.\par \medskip
 
 \noindent The proof for  (ii) is  similar to the previous one. In this case  the lattice $\Z L \oplus E_8(-2)$ has index 2 in $NS(X$ and we may assume that $NS(X)$ is generated by $L$, $E_8(-2)$ and $E_1 = (L+v)/2$, with $v \in E_8(-2)$, such that $v^2 =-4$. Then $E_1$ and $E_2$, where $E_2 =(L-v)/2$, are the classes of 2 elliptic fibrations. The map 
$$\Phi_L : X \to \P^3$$
\noindent is a 2:1 map to a quadric $Q$ in $\P^3$ and it is ramified on a curve $C$ of bidegree $(4,4)$ ([VG-S 3.5]). The quadric $Q$ is smooth, hence it is isomorphic to $\P^1 \times \P^1$. The covering involution $\sigma :X \to X $  of $X \to Q$ and the Nikulin involution $i$ commute, the elliptic pencils $E_1$ and $E_2$ are permuted by $i$ because $i^* L=L$ and $i^*v =-v$. $i$ induces an involution  $i_{Q}$ on $Q\simeq \P^1 \times \P^1$  which acts sending a point  $\{(s,t),(u,v)\} $ to $ \{ (u,v),(s,t)\} $ . The quotient $Q/<i_Q>$ is isomorphic to $\P^2$.  Let $j =i \circ \sigma =\sigma \circ i$  in $Aut \ (X)$ and let   $G = \{1,\sigma,i, j\} \simeq \Z/2\Z \times \Z/2\Z$. $S = X/<j>$ is a Del Pezzo surface of degree 2, by [VG-S 3.5]. The motives $h(Q)$ and $h(S)$ have no transcendental part, hence from the same argument  as in (i), we get an isomorphism  $t_2(X) \simeq t_2( Y)$.\par \medskip
\noindent We  now consider the description given in [VG-S 3.7] of  (iii). Let $Y$ be  the desingularization of the quotient  surface $X/<i>$ . In this case  there is a line bundle  $M \in NS( Y)$ such that $\beta^*L \simeq f^*M$  and 
$$H^0(X,L) \simeq f^*(H^0( Y,M)) \oplus f^*(H^0( Y, M-   C))$$
where  $\beta : \tilde X \to X$ is the blow-up at the 8 fixed points $ P_1,P_2,...P_8 $ of  $i$,  $f:  \tilde X \to Y$ and  $ C= (\sum_{ 1 \le i \le 8}   C_i) /2 \in NS( Y)$, with $ C_i$  the rational curves  on  $Y $ corresponding to the 8 singular points $Q_1,\cdots,Q_8$ of $ X/<i>$. The above decomposition is the decomposition of $H^0(X,L)$ into the $i^*$ eigenspaces. We have $L^2 =8$, $M^2 =4$, $h^0(M)=4$, $ h^0(M -C)=2$ so that
 $$ \Phi_L : X \to \P^5  \  ;  \  \Phi_M :   Y \to \P^3  \ ; \  \Phi_{M- C} :  Y \to \P^1           $$
\noindent The image of $X$ under $\Phi_L$ is the intersection of 3 quadrics  in $\P^5$ and the involution $i$ is induced by the involution
$$  \C^6 : (x_0,x_1,x_2,x_3,y_o,y_1) \to (x_0,x_1,x_2,x_3,-y_o,-y_1) $$
\noindent The fixed points   $ (P_1,P_2,...P_8)$ lie in  $X \cap \{ y_0 =y_1=0 \}$. The quadrics in the ideal of $X$ are of the form
$$y^2_0 =Q_1(x),   y_0y_1 =Q_2(x)  , y^2_1=Q_3(x)$$
\noindent where $x =(x_0,x_1,x_2,x_3)$. The line  
$$l : x_0=x_1=x_2=x_3=0$$
\noindent in $\P^5$ is fixed under $i$ and $l \cap X = \emptyset$.The image of $  Y$ by $\Phi_M$ is the projection of $X$ from the invariant line to the invariant $\P^3$ which is defined by $y_0=y_1=0$. The image is the quartic surface in $\P^3$ defined by
$$ Q_1(x)Q_3(x) -Q^2_2(x)=0$$
which can be identified with $Y$. \par 
\noindent We now use a result  in  [Vois 1.18] : if $X$ is the K3 surface obtained as the  intersection of 3 quadrics in $\P^5$ which are invariant under the involution 
$$i : (x_0,x_1,x_2,x_3,y_o,y_1) \to (x_0,x_1,x_2,x_3,-y_o,-y_1)$$
then $i^*(\omega ) = \omega$ for $\omega \in H^{2,0}(X)$ and $i$ acts trivially on $A_0(X)$. Therefore, by Corollary 1 (iii)     we get  an isomorphism $ \theta : t_2(X) \iso t_2( Y) $. 

\medskip 

Next we consider  the case  of a  K3 surface $X$  which has an elliptic fibration $g : X \to \P^1$  with a global section  $\sigma : \P^1 \to X$. The set of sections of $g$ is the Mordell-Weil group $MW_g$ with identity element $\sigma$. $MW_g$ is the subgroup  of Aut $X$ consisting of all automorphism acting  on a general fiber as translations and these translations preserve the holomorphic two form on $X$. Therefore, if there is an element  $\tau$ of order 2 in $MW_g$ then the translation by $\tau$ defines a Nikulin involution  $i$ on $X$.

\begin {thm} Let   $X$  a general  elliptic fibration $g : X \to \P^1$  with   sections   $\sigma$, $ \tau$ as above. Let  $i$ be the corresponding Nikulin involution on $X$ and let $Y$ be the desingularization of $X/<i>$.
Then the map $f : X \to X/<i>$ induces an isomorphism
$$\theta : t_2(X) \iso t_2(Y)$$
\end{thm} 
\begin{proof}In [VG-S 4.2] it  is shown that for  a general elliptic fibration $g: X \to \P^1$ there is an isomorphism  $MW_g  =\{\sigma,\tau \} \simeq \Z/2\Z$where $\sigma : \P^1\to X$ . Hence  the translation by $\tau$ defines a Nikulin involution  $i$ on $X$. The Weierstrass equation of $X$ can be put in the form 
$$ X :  y^2=x(x^2+a(t)x +b(t))$$
where the degree of $a(t)$ and $b(t)$ are 4 and 8 respectively. There are 8 singular fibers of type $I_1$, which are rational curves with a node, corresponding to the zeroes  $\{a_1,\cdots,a_8\}$ of $a^2(t) - 4b(t)$ and 8 singular fibers of type $I_2$, which are union of two $\P^1$  meeting in 2 points, corresponding to the zeroes  $  \{b_1,\cdots,b_8\}$ of $b(t)$.  The fixed points of the translation by $\tau$ are the 8 nodes in the $I_1$-fibers. $\tau$ acts on the generic fiber $E_t$ as the 
translation by a point $P$ of order 2, i.e. $\tau(Q) =Q +P$ with $2P=0$. The desingularization $ Y$ of the quotient surface $ X/<\tau>$  is 
an elliptic fibration with Weierstrass equatiion
$$  Y : y^2 = x(x^2- 2a(t)x+9a(t)^2-4b(t).$$ 
\noindent The generic fiber $ F_t$ of  $ Y$ is the elliptic curve 
$E_t/<P>$, where $E_t$ is the generic fiber on $X$.  Let
$$V =\bigcup_{t \in A} g^{-1}(t)= \bigcup E_t $$
\noindent  where $A = \P^1 - \{a_1,\cdots ,a_8, b_1,\cdots ,b_8 \} $. 
Then $V$ is open in $X$ and,   for every point $x \in V$,  the involution $\tau$ acts as  translation by a point of order 2 on $E_t$,  so that $2 \tau(x) = 2x$. Therefore $2 (1 \times \tau)(x,x) =(2x,2x)$, for all $x \in V$  i.e. $(1 \times \tau)\Delta_V =\Delta_V$  with  $ \Delta_V =\Delta_X \cap (V \times X)$. We get  $ (1\times \tau)\Delta_X -\Delta_X =0$ on $V \times V$, hence
$ (1\times \tau)\Delta_X -\Delta_X \in \sJ(X)$, with $\sJ(X) =Ker \ \Psi_X $ and  $\Psi_X :  A^2(X \times X) \to End_{\sM_{rat}}(t_2(X))$. 
Therefore  $\Psi_X(1 \times  \tau)\Delta_X )=   \Psi_X( \Delta_X )$ and 
 $$\bar \tau([\xi] )= [\xi]$$
\noindent in  $ End_{\sM_{rat}}(t_2(X)$, where $\xi$ is the generic point of $X$. From Corollary 1 (ii) we get
$$\theta :  t_2(X) \iso t_2( Y)$$  
\end{proof}
\begin {rk} By [VG-S 4.1],   if $X$ is as in theorem 7 ,  then the Neron-Severi group $NS(X)$ has rank $\rho(X)=10$, and  $dim \ T_{X,\Q}=12$ is even. In this case the isomorphism of Hodge structures $\phi_i: T_{X,\Q} \simeq T_{Y,\Q}$,  induced by the involution $i$,  is an isometry. On the contrary, in the cases described in (i),(ii), (iii), where $\rho(X)=9$,   $\phi_i$ is not an isometry. This follows from [VG-S 2,5] because $dim \  T_{X,\Q}$ is odd.  \end{rk}

 \end{document}